\documentclass[11pt]{amsart}
\usepackage[T1]{fontenc}
\usepackage[utf8]{inputenc}
\usepackage[english]{babel}
\usepackage[a4paper,margin=29mm]{geometry}
\usepackage{amsmath,amssymb,amsthm,mathtools}
\usepackage{tikz-cd}
\usepackage{enumitem}
\usepackage[colorlinks=true,linkcolor=blue,citecolor=blue,urlcolor=blue]{hyperref}
\hypersetup{
  pdftitle={n-Centers of Principal Subalgebras of Path Algebras},
  pdfauthor={Yihao Zheng and Shenglin Zhu}
}

\usepackage{etoolbox}
\makeatletter
\patchcmd{\@settitle}{\uppercasenonmath\@title}{}{}{}
\patchcmd{\@settitle}{\bfseries}{\bfseries\boldmath}{}{}
\makeatother

\newtheorem{lemma}{Lemma}[section]

\newtheorem{teo}[lemma]{Theorem}
\newtheorem{theorem}[lemma]{Theorem}
\newtheorem{proposition}[lemma]{Proposition}

\theoremstyle{definition}
\newtheorem{defi}[lemma]{Definition}
\newtheorem{ex}[lemma]{Example}

\newtheorem{example}[lemma]{Example}
\theoremstyle{remark}
\newtheorem{rem}[lemma]{Remark}

\setlist[enumerate,1]{label=(\roman*),ref=\roman*}

\newcommand{\kk}{\mathbb K}
\DeclareMathOperator{\St}{St}
\DeclareMathOperator{\Span}{span}

\DeclareMathOperator{\Mat}{M}

\begin{document}
\title
{$n$-Centers of Principal Subalgebras $\kk Q_{\geq 1}$ of Path Algebras}

\author{Yihao Zheng}
\address{School of Mathematical Sciences, Fudan University, Shanghai 200433, China}
\email{yhzheng24@m.fudan.edu.cn}

\author{Shenglin Zhu}
\address{School of Mathematical Sciences, Fudan University, Shanghai 200433, China}
\email{mazhusl@fudan.edu.cn}


\begin{abstract}

Over a field $\kk$ of characteristic zero, we determine the $\St_n$-elements and $n$-centers, defined by standard polynomials,  of principal subalgebras of path algebras.
Our description  expresses each $n$-center as the direct sum of a path ideal
and explicit subspaces associated with isolated circles.
We prove that these $n$-centers are subalgebras and stabilize
for $n\geq2|Q_0|$.
We also obtain explicit descriptions of the center and the
$3$-center, and determine the least degree of a standard
polynomial identity in the acyclic case.
Further applications concern algebras of block formal languages
and rational Hilbert series of the $n$-centers.

\end{abstract}

\maketitle

\section{Introduction}

Standard polynomials play an important role in the study of noncommutative algebra. Owing to their intrinsic alternating symmetry, multilinearity, and permutation-based construction, they have evolved from an algebraic construction into a powerful tool, with applications to noncommutative algebra, matrix theory, invariant theory, combinatorics, representation theory, noncommutative algebraic geometry, quantum groups, and beyond; see, for example, \cite{MR36751}, \cite{ART_2026__3_2_165_0},  \cite{MR302689}, \cite{MR1108620}, \cite{MR366968}, \cite{MR576061}.

Using standard polynomials, we introduce the \(\St_n\)-elements and \(n\)-centers of algebras.
We say an element $x_1$ of an algebra $A$ is an $\St_n$-element if for every $(n-1)$-tuple $(x_2, \cdots, x_n)$, $\St_n(x_1, \cdots, x_n) = 0$, where $n \geq 2$, $x_2, \cdots, x_n \in A$. The set of all $\St_n$-elements is denoted by $Z_n(A)$, which is called the $n$-center of $A$.  When $n=2$, the definition of an $\St_2$-element coincides with that of a central element, and $Z(A)=Z_2(A)$. Moreover, \(Z_n(A)=A\) if and only if \(\St_n\) is a polynomial identity of $A$.

The Amitsur--Levitzki theorem~\cite{MR36751} gives
$Z_{2r}(\Mat_r(\kk))=\Mat_r(\kk)$ and shows that $2r$
is the least degree for which this equality holds.
We determine the $n$-centers of full matrix algebras
in all degrees $n\geq2$.

\begin{theorem}[Lemma~\ref{lem:matrixkernel}]\label{thm:matrixzn}
Let $\kk$ be a field of characteristic zero.
For every $r\geq1$ and $n\geq2$,
\[
Z_n(\Mat_r(\kk))=
\begin{cases}
\kk I_r,&n<2r\text{ and }n\text{ is even},\\
0,&n<2r\text{ and }n\text{ is odd},\\
\Mat_r(\kk),&n\geq2r.
\end{cases}
\]
\end{theorem}

This result is used in the proof of our main theorem and explains
the distinction between even and odd degrees for
isolated circle components.

For path algebras, Corrales Garc\'{i}a et al.~\cite{MR3712583}
characterize the center of the path algebra of any finite quiver.
Cerulli Irelli, De Loera Ch\'avez, and
Pascucci~\cite{ART_2026__3_2_165_0}
characterize the finite quivers whose path algebras satisfy
a polynomial identity.
For such quivers, Berele et al.~\cite{MR5062649}
prove that subalgebras generated by paths are $PI$-equivalent
to associated incidence algebras.

In this paper, we determine the entire subspace
$Z_n(\kk Q_{\geq1})$ for every finite quiver $Q$
and every $n\geq2$, where $\kk Q_{\geq1}$ is the principal
subalgebra spanned by all non-trivial paths in $Q$.
With the notation introduced in Section~\ref{sec:main},
our main theorem is the following.

\begin{theorem}[Theorem~\ref{thm:main}]\label{thm:themain}
Let $Q$ be a finite quiver, and let $\operatorname{char}\kk=0$. For every $n\geq2$,
\[
Z_n(\kk Q_{\geq 1})=
\begin{cases}
\mathcal{I}_n,
& n\text{ is odd},\\[10pt]
\displaystyle
\mathcal{I}_n\oplus
\bigoplus_{\substack{
C\text{ isolated }r_C\text{-circle}\\
2r_C>n}}
\Span_{\kk}
\left\{
\sum_{i\in C_0}c_i^k \mid k\geq1
\right\},
& n\text{ is even}.
\end{cases}
\]
\end{theorem}

The theorem expresses each \(n\)-center as the direct sum of a path ideal and explicit subspaces supported on isolated circles. The additional subspaces occur only when \(n\) is even.

The proof combines the decomposition by starting and terminal
vertices with the grading by arrow multiplicities.
Noncircle cyclic components give rise to free subalgebras on at least
two generators and force the corresponding parts of the
$n$-center to vanish.
On subquivers whose components are trivial or circles,
we construct a faithful representation by matrices over
a commutative polynomial ring.
After passing to its field of fractions, the standard
polynomials can be computed on matrix units.
The Amitsur--Levitzki theorem, the description of
$Z_n(\Mat_r(\kk))$ in Theorem~\ref{thm:matrixzn},
and the order in which a path traverses the strongly connected components
give the exact bounds $L(C)+R(D)$ and $h(C)$.
These calculations also determine which linear combinations
of paths can belong to the $n$-center.

Several consequences follow from Theorem~\ref{thm:themain}. Every $Z_n(\kk Q_{\geq1})$ is a subalgebra of $\kk Q_{\geq1}$. If \(Q\) has no isolated circle with at least two vertices, these subalgebras form an increasing sequence of two-sided ideals.  If $N=|Q_0|$, then $Z_n(\kk Q_{\geq1})=Z_{2N}(\kk Q_{\geq1})$ for every $n\geq2N$. $Z_{2N}(\kk Q_{\geq1})$ is spanned by the non-trivial paths $p$ such that no noncircle cyclic component can reach $s(p)$ and no noncircle cyclic component can be reached from $t(p)$.  We obtain explicit descriptions of the center and the $3$-center. 
For acyclic quivers, the description uses only the maximum
lengths of paths ending at $s(p)$ and starting at $t(p)$,
and determines the least degree of a standard polynomial
identity of $\kk Q_{\geq1}$.

Further applications concern formal languages and Hilbert series. Algebras of block formal languages defined by finitely many forbidden words can be realized as principal subalgebras of  path algebras of finite quivers. Their $n$-centers are therefore described by the same theorem. Moreover, the paths spanning $\mathcal{I}_n$ form a regular language. Its length generating series is the Hilbert series of $\mathcal{I}_n$ and can be computed from the adjacency matrix of $Q$. Adding the explicit contributions from isolated circles gives a rational formula for the Hilbert series of every $Z_n(\kk Q_{\geq1})$.

The paper is organized as follows. In Section~\ref{sec:preliminaries}, we recall the basic definitions and facts concerning standard polynomials, \(n\)-centers, quivers, and path algebras. In Section~\ref{sec:main}, we introduce the combinatorial data associated with the strongly connected components of \(Q\) and state the main theorem. Since the proof of the main theorem is relatively long and requires several separate reductions, we give its statement and proof in two different sections, with Section~\ref{sec:proof} devoted to the proof. In Section~\ref{sec:consequences}, we derive consequences of the main theorem, including stabilization, applications to algebras of block formal languages and Hilbert series, and several explicit examples.

\section{Preliminaries}\label{sec:preliminaries}

\subsection{Standard polynomials and \texorpdfstring{$n$}{n}-centers}

In this subsection, we 
recall basic facts about standard polynomials and fix some notation.
 Standard references are \cite{MR4249615}, \cite{ART_2026__3_2_165_0}, and  \cite{MR2176105}. Throughout this paper, $\kk $ is a field of characteristic zero and all algebras are associative over $\kk $.

\begin{defi}
The standard polynomial of degree $n$ is defined to be
\[
\St_n(x_1, \cdots, x_n) = \sum_{\sigma \in S_n} \operatorname{sgn}(\sigma)\, x_{\sigma(1)} \cdots x_{\sigma(n)},
\]
where $S_n$ is the symmetric group of degree $n$ and $\operatorname{sgn}(\sigma)$ is the sign of the permutation $\sigma$.
\end{defi}

Standard polynomials satisfy the following Leibniz-type formula:
\[
\St_{n+1}(x_1, \cdots, x_{n+1}) = \sum_{i=1}^{n+1} (-1)^{i+1}\, x_i\, \St_n(x_1,\cdots,\hat{x}_i,\cdots,x_{n+1}).
\]

\begin{defi}
We say an element $x_1$ of an algebra $A$ is an $\St_i$-element if for every $(i-1)$-tuple $(x_2,\cdots,x_i)$, $\St_i(x_1,\cdots,x_i)=0$, where $i \geq 2$ and $x_2,\cdots,x_i \in A$. We define the $n$-center of $A$ to be the set of all $\St_n$-elements of $A$, denoted by $Z_n(A)$, where $n \ge 2$. In particular, $Z_2(A)=Z(A)$.  Moreover, \(Z_n(A)=A\) if and only if \(\St_n\) is a polynomial identity of $A$.
\end{defi}

Standard polynomials form an important class of multilinear polynomials. 
Recall that a polynomial $f(x_1,\cdots,x_m)$ is said to be multilinear if it is a
linear combination of monomials in which each variable
$x_1,\cdots,x_m$ occurs exactly once.

In 1950, Amitsur and Levitzki used standard polynomials to prove the famous Amitsur--Levitzki theorem, which demonstrated the crucial role of standard polynomials in $PI$ theory. In our notation, the  theorem can be restated as follows.

\begin{teo}[Amitsur--Levitzki]\cite{MR36751}\label{Thm:Amitsur-Levitski}
Let \(S\) be a commutative ring. Then, for every \(n\geq 1\), \(Z_{2n}(M_n(S))=M_n(S)\). Moreover, if \(S\) is a field, then \(2n\) is the least positive integer \(m\) for which \(Z_m(M_n(S))=M_n(S)\).
\end{teo}

\subsection{Quivers and path algebras}

In this subsection, we fix notation and recall some preliminaries about quivers and path algebras. Standard references are \cite{MR2197389}, \cite{MR3727119}, and \cite{MR3308668}.

A quiver $Q$ is a quadruple $Q = (Q_0, Q_1, s, t)$ where
\begin{itemize}
\item $Q_0$ is a finite set of vertices,
\item $Q_1$ is a finite set of arrows,
\item $s, t: Q_1 \to Q_0$ are two functions. For an arrow $a \in Q_1$, $s(a)$ and $t(a)$ are called the starting vertex and the terminal vertex of $a$, respectively.
\end{itemize}

A vertex \(v\in Q_0\) is called a source if there is no arrow \(a\in Q_1\) with \(t(a)=v\). It is called a sink if there is no arrow \(a\in Q_1\) with \(s(a)=v\).

\begin{rem}
Throughout this paper, for a quiver $Q$, we always assume that $Q_0$ and $Q_1$ are non-empty.
\end{rem}

\begin{ex}
A quiver whose underlying graph is the $A_3$ Dynkin graph is given by
\begin{center}
\begin{tikzcd}
1 \arrow[r, "\alpha"] & 2 \arrow[r, "\beta"] & 3
\end{tikzcd}
\end{center}
where $Q_0 = \{1,2,3\}$, $Q_1 = \{\alpha,\beta\}$, $s(\alpha)=1$, $t(\alpha)=2=s(\beta)$, $t(\beta)=3$.
\end{ex}

A path $p$ in a quiver $Q$ of length $n = \ell(p)$ (or $|p|$) $\ge 1$ is a sequence $p = \alpha_1 \alpha_2 \cdots \alpha_n$ of $n$ arrows such that $t(\alpha_i) = s(\alpha_{i+1})$ for $i=1,2,\cdots,n-1$; the starting vertex of $p$ is $s(p) = s(\alpha_1)$ and the terminal vertex of $p$ is $t(p) = t(\alpha_n)$. For every vertex $v \in Q_0$, we introduce a trivial path $e_v$ of length zero and define $s(e_v)=t(e_v)=v$. A closed path  in $Q$ is a  path $c$ with $t(c)=s(c)$, and it is called  an oriented cycle if it is non-trivial.  An oriented cycle with $n$ vertices $\{1,2,\cdots,n\}$ and $n$ arrows $\{a_1,a_2,\cdots,a_n\}$ satisfying $t(a_i)=s(a_{i+1})=i+1$ for $i=1,\cdots,n-1$ and $s(a_1)=t(a_n)=1$ is called an $n$-circle.  A loop is a $1$-circle. A quiver without oriented cycles is said to be acyclic.

Given a quiver $Q$ and a field $\kk $, the path algebra $\kk Q$ of $Q$ is the $\kk $-algebra with a basis labeled by all paths in $Q$, including trivial paths. The product of two paths $p$ and $q$ is the concatenation $pq$ if $t(p)=s(q)$ and zero otherwise. The multiplication in $\kk Q$ is defined by extending the path multiplication linearly. A finite quiver $Q$ is acyclic if and only if $\kk Q$ is finite-dimensional.

The principal subalgebra $\kk Q_{\ge 1}$ of $\kk Q$ is the subalgebra with a basis labeled by all non-trivial paths in $Q$. If $Q$ is a finite acyclic quiver, then $\kk Q_{\ge 1}$ is the Jacobson radical of $\kk Q$.

\section{The main theorem}\label{sec:main}

In this section, we give the statement of our main theorem and explain the meaning of each symbol appearing in it.

Let $Q$ be a finite quiver, and let $\mathcal S(Q)$ be the set of strongly connected components of  $Q$. A strongly connected component (or SCC, for short) $C$ of $Q$ is also a subquiver of $Q$ with vertex set $C_0$ and arrow set $C_1$. A circle component is an SCC  which is an $r$-circle for some $r\geq1$. When $r=1$, it consists of one vertex with exactly one loop. An SCC consisting of one vertex and no loop is called trivial. All  other SCCs are called noncircle cyclic components.

The condensed quiver of $Q$ has vertex set $\mathcal S(Q)$ and an arrow which goes from SCC $C$ to SCC $D$ whenever an arrow of $Q$ goes from $C$ to $D$ with $C\ne D$. It is acyclic. We write $C\preceq D$ if $C=D$ or there is a  path from $C$ to $D$ in the condensed quiver. We write $C\prec D$ if $C\preceq D$ and $C\ne D$. For \(v\in Q_0\), let \(C_v\) denote the SCC containing \(v\).
Define
\begin{equation*}
 w(C)=\begin{cases}
 1,&C\text{ is trivial},\\
 2r,&C\text{ is a circle component with }r\text{ vertices},\\
 \infty,&C\text{ is a noncircle cyclic component}.
 \end{cases}
\end{equation*}
For each component \(C\in\mathcal S(Q)\), define
\[
L(C)=\max_{C_1\prec\cdots\prec C_a=C}
       \sum_{j=1}^a w(C_j),
\qquad
R(C)=\max_{C=C_1\prec\cdots\prec C_a}
       \sum_{j=1}^a w(C_j),
\]
where \(C_1,\cdots,C_a\) are SCCs of \(Q\).

For $C,D\in\mathcal S(Q)$, let $A_{CD}=\Span_{\kk} \{p  \mid p\text{ is a non-trivial path}, s(p)\in C_0, \ t(p) \in D_0\}$.
Thus we have
$\kk Q_{\geq1}=\bigoplus_{C,D\in\mathcal S(Q)}A_{CD}$.
If $C$ is an $r_C$-circle, define
 $h(C)=L(C)+R(C)-2r_C$
with values in $\mathbb N\cup\{\infty\}$. For $i\in C_0$, let $c_i$ denote the shortest non-trivial path in $C$ starting at $i$ and ending at $i$. A circle component $C$ is isolated if it is an isolated vertex of the condensed quiver. In that case $L(C)=R(C)=2r_C$ and $h(C)=2r_C$.

For $n\geq2$, define 
\begin{equation*}
 \mathcal{I}_n=
 \bigoplus_{\substack{C\ne D\\L(C)+R(D)\leq n}}A_{CD}
 \ \oplus\!
 \bigoplus_{\substack{C\text{ circle}\\h(C)\leq n}}A_{CC}.
\end{equation*}

We can now state the main theorem.

\begin{theorem}\label{thm:main}
Let $Q$ be a finite quiver, and let $\operatorname{char}\kk=0$. For every $n\geq2$,

\begin{equation}\label{eq:main}
Z_n(\kk Q_{\geq 1})=
\begin{cases}
\mathcal{I}_n,
& n\text{ is odd},\\[10pt]
\displaystyle
\mathcal{I}_n\oplus
\bigoplus_{\substack{
C\text{ isolated }r_C\text{-circle}\\
2r_C>n}}
\Span_{\kk}
\left\{
\sum_{i\in C_0}c_i^k \mid k\geq1
\right\},
& n\text{ is even}.
\end{cases}
\end{equation}
\end{theorem}

\section{Proof of the main theorem}\label{sec:proof}
We begin by outlining the proof of Theorem~\ref{thm:main}. The first step is to decompose elements of \(Z_n(\kk Q_{\geq1})\) according to their starting and terminal vertices and their arrow multidegrees. This allows us to treat the different SCCs separately and, in particular, to exclude paths contained in noncircle cyclic components and in circle components comparable with a noncircle cyclic component.
We then consider paths whose starting and terminal vertices lie in distinct SCCs. 
After embedding the corresponding path algebras into matrix algebras over commutative polynomial rings, the quantities \(L(C)\) and \(R(D)\) give the exact condition \(L(C)+R(D)\leq n\) for such paths to belong to the \(n\)-center. 
The remaining case consists of paths contained in circle components. For these components, we first determine the \(n\)-centers of full matrix algebras and then combine this calculation with the possible paths entering and leaving the circle. 
This gives the condition \(n\geq h(C)\) under which a nonisolated circle component \(C\) contributes \(A_{CC}\) to the \(n\)-center, while isolated circles give the additional subspaces occurring in even degrees. The results of these steps are assembled in Subsection~\ref{sub:completion} to complete the proof of Theorem~\ref{thm:main}.

\subsection{Decomposition of elements and noncircle components}\label{sec:separation}
In this subsection, we study the decomposition of elements of
$Z_n(\kk Q_{\geq1})$ by their starting and terminal vertices
and by arrow multidegree. We then show that paths contained
in noncircle cyclic components, or in circle components
comparable with a noncircle cyclic component, have zero coefficients in every element
of $Z_n(\kk Q_{\geq1})$.

\begin{lemma}\label{lem:derivation}
Let $A$ be an associative $\kk$-algebra.
\begin{enumerate}
\item If $\delta:A\to A$ is a derivation, then
$\delta(Z_n(A))\subseteq Z_n(A)$.
\item If $\varphi:A\to A$ is an automorphism, then
$\varphi(Z_n(A))=Z_n(A)$.
\end{enumerate}
\end{lemma}

\begin{proof}
If $\delta$ is a derivation, we have 
$\delta(\St_n(a_1,\cdots,a_n))
=\sum_{j=1}^n
\St_n(a_1,\cdots,\delta(a_j),\cdots,a_n)$.
If $a_1\in Z_n(A)$,  the left side and all terms
with $j\geq2$ vanish. Hence
$\St_n(\delta(a_1),a_2,\cdots,a_n)=0$, which shows that
$\delta(a_1)\in Z_n(A)$.

Now let $\varphi:A\to A$ be an automorphism.
For $z\in Z_n(A)$ and arbitrary $a_2,\cdots,a_n\in A$,
we have
$\St_n(\varphi(z),a_2,\cdots,a_n)
=\varphi(\St_n(z,\varphi^{-1}(a_2),\cdots,
\varphi^{-1}(a_n)))=0$.
Thus $\varphi(Z_n(A))\subseteq Z_n(A)$.
Applying the same inclusion to $\varphi^{-1}$ gives
$\varphi^{-1}(Z_n(A))\subseteq Z_n(A)$, and therefore
$Z_n(A)\subseteq\varphi(Z_n(A))$.
Hence $\varphi(Z_n(A))=Z_n(A)$.
\end{proof}

For $i\in Q_0$, $D_i=[e_i,-]$ maps $\kk Q_{\geq 1}$ to $\kk Q_{\geq 1}$ and is a derivation, although $e_i$ does not belong to $\kk Q_{\geq 1}$. For a path $p$,
$D_i(p) \in \{p, -p, 0\}$.
For $i\ne j$, the operator $-D_j(D_i^2+D_i)/2$ is the projection onto $e_i\kk Q_{\geq 1}e_j$. So, if $z \in Z_n(\kk Q_{\geq 1})$, then $e_ize_j \in Z_n(\kk Q_{\geq 1}) $.

For each arrow \(a\in Q_1\) and each path \(p\), let
$\deg_a(p)$
denote the number of occurrences of \(a\) in \(p\). For
$\mathbf d=(d_a)_{a\in Q_1}\in \mathbb N^{Q_1}$,
let \(A_{\mathbf d}\) be the \(\kk \)-span of all  non-trivial paths \(p\) satisfying
$\deg_a(p)=d_a$ for every $a\in Q_1$.
Then
$\kk Q_{\geq 1}=\bigoplus_{\mathbf d\in\mathbb N^{Q_1}}A_{\mathbf d}$.
Moreover,
$Z_n(\kk Q_{\geq 1})=
\bigoplus_{\mathbf d\in\mathbb N^{Q_1}}
(Z_n(\kk Q_{\geq 1})\cap A_{\mathbf d})$.

Indeed, fix an arrow \(a\in Q_1\). For every \(t\in\kk ^\times\), define
$$\varphi_{a,t}(a)=ta,
\qquad
\varphi_{a,t}(b)=b
\quad (b \neq a).$$
This map extends  to an automorphism
$\varphi_{a,t}:\kk Q_{\geq 1}\longrightarrow \kk Q_{\geq 1}$.
By Lemma~\ref{lem:derivation},
$\varphi_{a,t}(Z_n(\kk Q_{\geq 1}))=Z_n(\kk Q_{\geq 1})$. Let \(z\in Z_n(\kk Q_{\geq 1})\), and decompose \(z\) according to the number of occurrences of the arrow \(a\). We have
$z=z_0+z_1+\cdots+z_N$,
where every path occurring in \(z_d\) contains \(a\) exactly \(d\) times. Then
$\varphi_{a,t}(z)
=
\sum_{d=0}^{N}t^d z_d
\in Z_n(\kk Q_{\geq 1})$.
Since $\kk$ is infinite, choose distinct
$t_0,\cdots,t_N\in\kk ^\times$.
The Vandermonde matrix
$(t_i^d)_{0\le i,d\le N}$
is invertible. Hence each \(z_d\) is a linear combination of
$\varphi_{a,t_0}(z),\cdots,\varphi_{a,t_N}(z)$. Since \(Z_n(\kk Q_{\geq 1})\) is a linear subspace of \(\kk Q_{\geq 1}\), it follows that $z_d\in Z_n(\kk Q_{\geq 1})$ for $0\le d\le N$.
Applying this argument successively to every arrow \(a\in Q_1\) shows that every arrow-multihomogeneous component \(z_{\mathbf d}\) of \(z\) belongs to \(Z_n(\kk Q_{\geq 1})\).

Similarly, if
$A^{(m)}
=
\operatorname{span}_{\kk}
\{p \mid \lvert p\rvert=m\}$,
then
$\kk Q_{\geq 1}=\bigoplus_{m\ge1}A^{(m)}$
and
$Z_n(\kk Q_{\geq 1})
=
\bigoplus_{m\ge1}
(Z_n(\kk Q_{\geq 1})\cap A^{(m)})$.

This decomposition does not, however, separate distinct paths with the same arrow multidegree. For $\boldsymbol\lambda=(\lambda_a)_{a\in Q_1}\in(\kk^\times)^{Q_1}$, let $\varphi_{\boldsymbol\lambda}:\kk Q_{\geq 1}\to \kk Q_{\geq 1}$ be the automorphism defined by $\varphi_{\boldsymbol\lambda}(a)=\lambda_a a$ for every $a\in Q_1$. If \(p\) and \(q\) satisfy
$\deg_a(p)=\deg_a(q)$
for every $a\in Q_1$,
then for every such automorphism, we have 
$\varphi_{\boldsymbol\lambda}(p)
=
(\prod_{a\in Q_1}\lambda_a^{\deg_a(p)})p$ and
$\varphi_{\boldsymbol\lambda}(q)
=
(\prod_{a\in Q_1}\lambda_a^{\deg_a(p)})q$.
In particular,
$z_{\mathbf d}\in Z_n(\kk Q_{\geq 1})$
does not imply that each path occurring in \(z_{\mathbf d}\) belongs  to \(Z_n(\kk Q_{\geq 1})\). The matrix argument used later is needed precisely to handle this remaining issue.

To treat noncircle cyclic components, we first consider
free associative algebras without constant term.

\begin{lemma}\label{lem:free}
Let $F_+$ be the free associative algebra without constant term on at least two generators. Then $Z_n(F_+)=0$ for every $n\geq2$.
\end{lemma}
\begin{proof}
Let $0\ne z\in F_+$, and let $d$ be the maximal length of the words occurring in $z$. Choose distinct generators $a$, $b$ and distinct integers $N_2,\cdots,N_n>d$. Let $x_j=ba^{N_j}b$. Write \(z=\sum_u \lambda_u u\), where each \(u\) is a word of length at most \(d\). Choose a word \(w\) with \(\lambda_w\neq0\), and set
$W=wx_2\cdots x_n=w(ba^{N_2}b)\cdots(ba^{N_n}b)$.

Since \(N_j>d\) for every \(j\), no word occurring in \(z\) contains a power of \(a\) of length \(N_j\). Hence the  powers of \(a\) in \(W\) whose lengths are greater than \(d\) are precisely \(a^{N_2},\cdots,a^{N_n}\), appearing in this order.

Suppose that a product \(u_{\sigma(1)}\cdots u_{\sigma(n)}\) occurring in the expansion of \(\St_n(z,x_2,\cdots,x_n)\) is equal to \(W\), where \(u_1\) is a word occurring in \(z\) and \(u_j=x_j\) for \(j\ge2\). The ordered sequence of  \(a\)-powers of length greater than \(d\) in this product must  have lengths \(N_2,\cdots,N_n\). Since the integers \(N_2,\cdots,N_n\) are pairwise distinct,  \(x_2,\cdots,x_n\) must occur in exactly that order.

It remains to determine the position of \(u_1\). If \(u_1\) were not the first factor, then the first factor would be some \(x_j=ba^{N_j}b\), so the first  \(a\)-power of length greater than \(d\) would begin immediately after the initial  \(b\). In \(W=wba^{N_2}b\cdots ba^{N_n}b\), however, the first such power begins only after  \(wb\). This is impossible. Hence \(u_1\) must be the first factor.

Therefore equality with \(W\) forces \(u_1=w\) and \(\sigma=\mathrm{id}\). Thus \(W\) occurs in \(\St_n(z,x_2,\cdots,x_n)\) exactly once. Its coefficient is \(\lambda_w\neq0\). Hence
$\St_n(z,x_2,\cdots,x_n)\neq0$.
\end{proof}

\begin{lemma}\label{lem:bad}
If $C$ is a noncircle cyclic component and $v\in C_0$, then $e_v\kk Q_{\geq 1}e_v$ is a free algebra without constant term on at least two generators. Every element of $Z_n(\kk Q_{\geq 1})$ has zero component on all paths contained in $C$.
\end{lemma}
\begin{proof}
 $e_v\kk Q_{\geq 1}e_v$ is free on the 
 paths that meet \(v\) only at their starting and terminal vertices.
 Since $C$ is a noncircle cyclic component and $v\in C_0$, there are at least two such paths. 

 If \(z\in Z_n(\kk Q_{\geq 1})\), then   $0=e_v\St_n(z,x_2,\cdots,x_n)e_v
 =\St_n(e_vze_v,x_2,\cdots,x_n)$ for arbitrary \(x_2,\cdots,x_n\in e_v\kk Q_{\geq 1}e_v\).
Hence \(e_vze_v\in Z_n(e_v\kk Q_{\geq 1}e_v)\). By Lemma~\ref{lem:free}, \(Z_n(e_v\kk Q_{\geq 1}e_v)=0\), and therefore \(e_vze_v=0\). For $i\ne j$ in $C_0$, the projection $z_{ij}=e_ize_j$ belongs to $Z_n(\kk Q_{\geq 1})$. Choose a path $q$ from $j$ to $i$. Then $[z_{ij},q]\in Z_n(\kk Q_{\geq 1})$, and  $e_i[z_{ij},q]e_i=z_{ij}q$. 
Hence \(z_{ij}q=0\), and therefore \(z_{ij}=0\).
\end{proof}

\begin{lemma}\label{lem:comparable}
Let $C$ be a circle component. If there is a noncircle cyclic component $D$ satisfying $C \prec D$ or $D \prec C$, then every element of $Z_n(\kk Q_{\geq 1})$ has zero component on paths  contained in $C$.
\end{lemma}
\begin{proof}
Suppose that $z\in Z_n(\kk Q_{\geq1})$ has a nonzero coefficient
on a path contained in $C$.
If $e_ize_j\ne0$ for some distinct $i,j\in C_0$,
replace $z$ by $e_ize_j$, which still belongs to
$Z_n(\kk Q_{\geq1})$.
Otherwise, some non-trivial closed path in $C$ has
a nonzero coefficient in $z$.
Choose $k\geq1$ such that
$\sum_{i\in C_0}\lambda_i c_i^k\ne0$,
where $\lambda_i$ is the coefficient of $c_i^k$ in $z$.
This sum is an arrow-multihomogeneous component of $z$,
so it belongs to $Z_n(\kk Q_{\geq1})$.
Replace $z$ by this sum.

Suppose a noncircle cyclic component $D$ satisfies $D \prec C$. Choose $i \in C_0$ such that $e_iz$ is nonzero. Choose $v\in D_0$. Since $D \prec C$, there exists a non-trivial path $q$ from $v$ to $i$. For $n\geq4$, choose $n-2$ non-trivial closed paths $u_1,\cdots,u_{n-2}$ at $v$ with $\St_{n-2}(u_1,\cdots,u_{n-2})\ne0$ by Lemma~\ref{lem:free} and Lemma~\ref{lem:bad}. For $n=3$, choose one non-trivial closed path  $u_1$ at $v$. For $n=2$, use no such path. Consider the value of $\St_n$ on these paths, $q$, and $z$. In every nonzero product,  all closed paths at $v$ must precede $q$, and $z$ must follow $q$. So, for $n \ge 4$, this value is equal to $
 \St_{n-2}(u_1,\cdots,u_{n-2})qz$, up to sign. For \(n=3\) and \(n=2\), the corresponding values are \(u_1qz\) and \(qz\), up to sign. In each case,  the value is nonzero. This is a contradiction. The proof of the case where $C \prec D$ is similar. Therefore,  every element of $Z_n(\kk Q_{\geq 1})$ has zero component on paths  contained in $C$.

\end{proof}

\subsection{Paths joining distinct SCCs}\label{sec:capacity}
In this subsection, we construct a matrix representation for
path algebras whose SCCs are trivial
or circles. Using this representation and the bounds given
by $L$ and $R$, we determine which linear combinations of
paths joining distinct SCCs belong
to $Z_n(\kk Q_{\geq1})$.

\begin{lemma}\label{lem:faithful}
Suppose every SCC of a finite quiver $H$ is trivial or a circle. Let $\mathcal{R}_H=\kk[a  \mid  a\in H_1]$ be a commutative polynomial ring and $F_H=\operatorname{Frac}(\mathcal{R}_H)$ its field of fractions. The map
\[
 \Phi(a)=aE_{s(a),t(a)},\qquad \Phi(e_i)=E_{ii},
\]
defines an injective map $\kk H\to\Mat_{|H_0|}(\mathcal{R}_H)$. The $F_H$-linear span of $\Phi(\kk H_{\geq1})$ is
\[
 B_H(F_H)=\Span_{F_H}\{E_{ij} \mid \text{there is a non-trivial path from }i\text{ to }j\text{ in }H\}.
\]
For $z\in\kk H_{\geq1}$,
 $z\in Z_n(\kk H_{\geq1})$ if and only if
 $\Phi(z)\in Z_n(B_H(F_H))$.
\end{lemma}

\begin{proof}
For a non-trivial path $p=a_1\cdots a_m$ from \(i\) to \(j\), we have $\Phi(p)=(a_1\cdots a_m)E_{ij}$, where on the right \(a_1\cdots a_m\) is regarded as a monomial in the
commutative polynomial ring \(\mathcal{R}_H\). For a trivial path \(e_i\), $\Phi(e_i)=E_{ii}$.

We first prove that \(\Phi\) is injective. Since every SCC of  $H$ is trivial or a circle, it is easy to check that a path in \(H\)
is uniquely determined by its starting vertex and the multiplicity of each arrow in the path.

Now suppose that \(p\) and \(q\) are two paths with $\Phi(p)=\Phi(q)$. Then \(p\) and \(q\) have the same starting vertex, the same terminal vertex, and the same arrow multidegree. Hence \(p=q\). Moreover, for each pair \(i,j\), distinct paths from \(i\) to \(j\)
give distinct monomials multiplying \(E_{ij}\). Since distinct monomials
are linearly independent over $\kk$ in \(\mathcal{R}_H\), and different pairs \((i,j)\)
correspond to different matrix units, the images of the paths of
\( H\) are linearly independent over $\kk$. Hence \(\Phi\) is injective.

For every pair \(i,j\) for which there exists a non-trivial path
\(p\) from \(i\) to \(j\), we have $\Phi(p)=pE_{ij}$ where $p\in \mathcal{R}_H$. Since $p$ is invertible in
\(F_H\),
$E_{ij}=p^{-1}\Phi(p)$. Therefore, $\Span_{F_H}\Phi(\kk H_{\geq1})
=
\Span_{F_H}\{E_{ij}\mid \text{there is a non-trivial path from }i\text{ to }j\}
=
B_H(F_H)$.

Finally, let \(z\in\kk H_{\geq1}\). Suppose first that $z\in Z_n(\kk H_{\geq1})$.
For arbitrary \(x_2, \cdots, x_n\in\kk H_{\geq1}\),
$\St_n(\Phi(z),\Phi(x_2),\cdots,\Phi(x_n))
=
\Phi(\St_n(z,x_2,\cdots,x_n))
=
0$. Since \(\St_n\) is multilinear and \(B_H(F_H)\) is the \(F_H\)-linear span
of \(\Phi(\kk H_{\geq1})\), we have $\Phi(z)\in Z_n(B_H(F_H))$.
Conversely, suppose that
$\Phi(z)\in Z_n(B_H(F_H))$.
Then for arbitrary \(x_2,\cdots,x_n\in\kk H_{\geq1}\),
$0
=
\St_n(\Phi(z),\Phi(x_2),\cdots,\Phi(x_n))
=
\Phi(\St_n(z,x_2,\cdots,x_n))$. 
By the injectivity of \(\Phi\), \(\St_n(z,x_2,\ldots,x_n)=0\).
Hence
$
z\in Z_n(\kk H_{\geq1})$.

\end{proof}

The next lemma gives an interpretation of $L(C)$ and $R(C)$
in terms of nonzero standard polynomial values with a
prescribed terminal or starting vertex in $C$.

We use  Theorem~\ref{Thm:Amitsur-Levitski}, which states that
$
\St_{2r}=0
$
on \(\Mat_r(\kk)\). Consequently,
$
\St_m=0
$
on \(\Mat_r(\kk)\) for every \(m\ge 2r\).

\begin{lemma}\label{lem:capacity}
Let \(C\) be an SCC of \(Q\), and fix $i\in C_0$. For every $m\geq2$, there exist non-trivial paths $q_1,\cdots,q_m$ with $\St_m(q_1,\cdots,q_m)e_i\ne0$ if and only if $m<L(C)$. The corresponding condition for $e_i\St_m(q_1,\cdots,q_m)\ne0$ is $m<R(C)$.
\end{lemma}
\begin{proof}
For finite $L(C)$, let \(H\) be the full subquiver of \(Q\) on the vertex set \(\bigcup_{D\preceq C}D_0\). By Lemma~\ref{lem:faithful} and multilinearity, it is enough to take
\(q_1,\cdots,q_m\) to be matrix units belonging to \(B_H(F_H)\). Write
$q_\nu=E_{u_\nu v_\nu}$.
Consider a permutation \(\sigma\in S_m\) for which
$
q_{\sigma(1)}\cdots q_{\sigma(m)}E_{ii}\neq0$.
Since
$
E_{uv}E_{xy}\neq0$ if and only if $
v=x$,
the nonzero product determines a directed sequence
$
u_{\sigma(1)}
\longrightarrow
v_{\sigma(1)}=u_{\sigma(2)}
\longrightarrow\cdots\longrightarrow
v_{\sigma(m)}=i$.
Passing to SCCs and deleting consecutive
repetitions gives a chain
$
C_1\prec C_2\prec\cdots\prec C_b=C
$
in the condensed quiver. 

For this nonzero product, each \(q_\nu\) is of one of the following two
types. It is either contained in one of the SCCs \(C_j\), or it crosses from
one SCC of the chain to a later one. Once the product leaves a
component \(C_j\), it can never return to \(C_j\). Consequently, after
fixing the SCCs \(C_1,\cdots,C_b\) visited by the product, the matrix units which cross from one SCC to another must occur
in the order determined by the chain
$C_1\prec C_2\prec\cdots\prec C_b$.
Moreover, all matrix units contained in \(C_j\) must occur after the
product enters \(C_j\) and before it leaves \(C_j\). Thus every nonzero
product with this fixed chain has the form
\[
\begin{aligned}
&(\text{product of matrix units contained in }C_1)\,y_1\\
&(\text{product of matrix units contained in }C_2)\,y_2\\
&\cdots\\
&y_{b-1}
(\text{product of matrix units contained in }C_b),
\end{aligned}
\]
where \(y_j\) is the matrix unit crossing from \(C_j\) to \(C_{j+1}\).

For \(1\leq j\leq b\), let \(x_{j,1},\cdots,x_{j,m_j}\) be the matrix units contained in \(C_j\). Among nonzero permutation products, only the order of the matrix units inside each \(C_j\) may vary. Put $X_j=\St_{m_j}(x_{j,1},\cdots,x_{j,m_j})$ when $m_j\geq2$, and $X_j=x_{j,1}$ when $m_j=1$. Hence the sum of all terms of \(\St_m(q_1,\cdots,q_m)E_{ii}\)  is, up to  sign,
$$
X_1y_1X_2y_2\cdots y_{b-1}X_bE_{ii},
$$
where $X_j$ is omitted when $m_j=0$. 

If \(C_j\) is trivial, then no non-trivial path is contained in \(C_j\), so \(m_j=0=w(C_j)-1\). If \(C_j\) is a circle with \(r_j\) vertices, the matrix units \(E_{uv}\) with \(u,v\in C_j\) belong to the diagonal block indexed by \(C_j\), and this block is naturally identified with \(\Mat_{r_j}(F_H)\).
By Theorem~\ref{Thm:Amitsur-Levitski}, \(\St_s=0\) on \(\Mat_{r_j}(F_H)\) for every \(s\geq 2r_j\). Hence
a nonzero contribution can occur only if
\(m_j\leq 2r_j-1=w(C_j)-1\).  Since there are \(b-1\) crossing matrix units, this sum can be nonzero only if
$
m=\sum_{j=1}^b m_j+(b-1)
   \leq \sum_{j=1}^b w(C_j)-1$.
 Therefore \(m\leq L(C)-1\). Hence \(\St_m(q_1,\cdots,q_m)E_{ii}=0\) whenever \(m\geq L(C)\).

For existence, suppose $L(C)\geq3$. Choose a chain \(D_1\prec \cdots\prec D_s=C\) of SCCs such that \(L(C)=\sum_{\ell=1}^s w(D_\ell)\). Since \(L(C)<\infty\), every \(D_\ell\) is trivial or a circle. Suppose that \(D_\ell\) is a circle with \(r_\ell\) vertices. Order its vertices as \(v_1,\cdots,v_{r_\ell}\) along the orientation of the circle, and take the \(2r_\ell-1\) matrix units $
E_{v_1v_1},E_{v_1v_2},E_{v_2v_2},\cdots,
E_{v_{r_\ell-1}v_{r_\ell}},E_{v_{r_\ell}v_{r_\ell}}$.
Their product in the displayed order is nonzero, whereas every other permutation gives zero. If \(D_\ell\) is trivial, choose no matrix unit inside \(D_\ell\) and regard its unique vertex as both its first and its last vertex.   
For each \(1\leq\ell<s\), choose a matrix unit joining the last chosen vertex in \(D_\ell\) to the first chosen vertex in \(D_{\ell+1}\). Such a matrix unit belongs to \(B_H(F_H)\), since \(D_\ell\prec D_{\ell+1}\). In the final SCC \(C\), choose the ordering such that the last chosen vertex is \(i\).

The total number of chosen matrix units is
$
\sum_{\ell=1}^s(w(D_\ell)-1)+(s-1)
=\sum_{\ell=1}^s w(D_\ell)-1
=L(C)-1
$.
By construction, exactly one permutation of these matrix units has a nonzero product, and that product ends at \(i\). 
Hence the value of \(\St_{L(C)-1}\) on these matrix units remains nonzero after multiplication by \(E_{ii}\).
For each \(2\leq m<L(C)\), the last \(m\) matrix units in this ordering also have the property that exactly one permutation has a nonzero product, and this product ends at \(i\).
 Therefore, for every \(2\le m<L(C)\), there exist matrix units \(q_1,\cdots,q_m\in B_H(F_H)\) such that \(\St_m(q_1,\cdots,q_m)E_{ii}\neq0\).

For each \(q_\nu=E_{u_\nu v_\nu}\), choose a non-trivial path \(p_\nu\) from \(u_\nu\) to \(v_\nu\). Write \(\Phi(p_\nu)=\lambda_\nu q_\nu\), where \(0\neq\lambda_\nu\in \mathcal{R}_H\). Since the \(\lambda_\nu\) commute,  we have 
$
\Phi(\St_m(p_1,\cdots,p_m))
=
(\prod_{\nu=1}^m\lambda_\nu)
\St_m(q_1,\cdots,q_m)$.
The right side remains nonzero after multiplication by \(E_{ii}\). Since \(\Phi\) is injective, \(\St_m(p_1,\cdots,p_m)e_i\neq0\). Thus the required paths exist for every \(2\le m<L(C)\).

Now suppose that \(L(C)=\infty\). Then there exists a noncircle cyclic SCC \(D\) that can reach \(C\). 
If \(C\) is a noncircle cyclic component,  by Lemma~\ref{lem:free} and Lemma~\ref{lem:bad}, for every \(m\ge2\) there exist non-trivial closed paths \(u_1,\cdots,u_m\) at \(i\) such that \(\St_m(u_1,\cdots,u_m)\ne0\). Since all these paths end at \(i\), multiplication by \(e_i\) leaves their standard polynomial unchanged.

Otherwise, choose a noncircle cyclic component \(D\ne C\) that can reach \(C\), 
a vertex  \(v\in D_0\) and a non-trivial path \(q\) from \(v\) to \(i\). For $m\geq3$, Lemma~\ref{lem:free} and Lemma~\ref{lem:bad} give non-trivial closed paths $u_1,\cdots,u_{m-1}$ at $v$ such that $\St_{m-1}(u_1,\cdots,u_{m-1})\ne0$. Since $q$ must occur after all the $u_j$ in every nonzero product, $\St_m(u_1,\cdots,u_{m-1},q)e_i=\St_{m-1}(u_1,\cdots,u_{m-1})q\ne0$. For $m=2$, take any non-trivial closed path $u$ at $v$; then $\St_2(u,q)e_i=uq\ne0$.

The proof of the assertions concerning \(R(C)\) is similar.

\end{proof}

For a path joining distinct SCCs, we apply the preceding
lemma at its starting and terminal vertices. Combining
the two bounds gives the following description of
$Z_n(\kk Q_{\geq1})\cap e_i\kk Q_{\geq1}e_j$.

\begin{proposition}\label{prop:cross}
Let $i, j$ lie in different SCCs $C, D$. Then
$$Z_n(\kk Q_{\geq 1})\cap e_i\kk Q_{\geq 1}e_j=
\begin{cases}
 e_i\kk Q_{\geq 1}e_j,&L(C)+R(D)\leq n,\\
 0,&L(C)+R(D)>n.
 \end{cases}$$
\end{proposition}
\begin{proof}
We only need to consider the case where there exists a path from $i$ to $j$. Let $g$ be a non-trivial path from $i$ to $j$.  Since $i,j$ lie in different SCCs, no non-trivial path can occur before $g$ in one nonzero term  and after $g$ in another. Thus the tuple admitting a nonzero term in the expansion of the standard polynomial  has a unique partition into $l$ left elements $q_1, \cdots, q_l$, $r$ right elements $u_1, \cdots, u_r $,  and $g$, where $l+r=n-1$. Put $X=\St_l(q_1,\cdots,q_l)e_i$ when $l\geq2$, $X=q_1e_i$ when $l=1$, and $X=e_i$ when $l=0$. Similarly, put $Y=e_j\St_r(u_1,\cdots,u_r)$ when $r\geq2$, $Y=e_ju_1$ when $r=1$, and $Y=e_j$ when $r=0$. The standard polynomial value is, up to sign,  $XgY$.
The same partition works for every path with starting vertex $i$ and terminal vertex $j$.

If $n\geq L(C)+R(D)$, then either $l\geq L(C)$ or $r\geq R(D)$, so the corresponding $X$ or $Y$ is zero by Lemma~\ref{lem:capacity} or, for a single path, by the definition of $L(C)$ or $R(D)$. Conversely, if $n<L(C)+R(D)$, choose $l,r\geq0$ with $l+r=n-1$, $l<L(C)$ and $r<R(D)$. 
Choose the paths so that \(X,Y\ne0\), using Lemma~\ref{lem:capacity} when \(l\ge2\) or \(r\ge2\), and choosing any path ending at $i$ or starting at $j$ when \(l=1\) or \(r=1\). 
Such a path exists when $L(C)>1$ or $R(D)>1$, respectively. Let  $z\in e_i\kk Q_{\geq 1}e_j$ be a nonzero element which is homogeneous in total length.
Write \(X=\sum_p\alpha_p p\), \(z=\sum_s\beta_s s\), and \(Y=\sum_t\gamma_t t\) where $p,s,t$ are paths.  Since \(z\) is homogeneous in total length, all paths \(s\) occurring in \(z\) have the same length. Similarly, all paths occurring in \(X\) have length \(\sum_{\nu=1}^l|q_\nu|\), and all paths occurring in \(Y\) have length \(\sum_{\nu=1}^r|u_\nu|\). Hence, if \(pst=p's't'\), then cutting the common path at the two fixed positions gives \(p=p'\), \(s=s'\), and \(t=t'\). Thus distinct triples \((p,s,t)\) give distinct  paths in \(XzY\), so no cancellation occurs. Since \(X\), \(z\), and \(Y\) are all nonzero,  \(XzY\neq0\). Hence $z\notin Z_n(\kk Q_{\geq 1})$. Total-length decomposition reduces an arbitrary nonzero element of $e_i\kk Q_{\geq 1}e_j$ to this case.

\end{proof}

\subsection{Circle components}\label{sec:circle}
In this subsection, we determine which linear combinations
of paths contained in a circle component belong to
$Z_n(\kk Q_{\geq1})$. We first treat nonisolated circle
components with finite $L$- and $R$-values, and then
consider isolated circle components.

Let $C$ be a circle component with $L(C),R(C)<\infty$. Let \(H_C\) be the full subquiver on the vertices lying in SCCs comparable with \(C\) under \(\preceq\).  Every SCC of $H_C$ is trivial or a circle. Set $\mathcal{R}_{H_C}=\kk[a\mid a\in(H_C)_1]$ and $F_{H_C}=\operatorname{Frac}(\mathcal{R}_{H_C})$. Write
$
 B_{H_C}(F_{H_C})=\Span_{F_{H_C}}\{E_{uv} \mid \text{there is a non-trivial path from }u\text{ to }v\text{ in }H_C\}$.
The subspace
$
 M_{H_C}(F_{H_C})=\Span_{F_{H_C}}\{E_{ij} \mid i,j\in C_0\}
$
is a full matrix algebra $\Mat_r(F_{H_C})$, where $r=|C_0|$.

\begin{lemma}\label{lem:local}
If $z$ is a linear combination of paths contained in $C$, then $
 z\in Z_n(\kk Q_{\geq 1})$ if and only if $
 z\in Z_n(\kk H_C{}_{\geq1})$.
\end{lemma}
\begin{proof}
Assume $z\in Z_n(\kk H_C{}_{\geq1})$. In a nonzero product containing a path contained in $C$, every vertex visited before that path can reach $C$, and every vertex visited after it can be reached from $C$. Hence every path occurring in such a product is contained in \(H_C\). Therefore, if one of the remaining paths is not contained in \(H_C\), every term in the expansion of the corresponding standard polynomial is zero. If all remaining paths are contained in \(H_C\), the standard polynomial vanishes because \(z\in Z_n(\kk H_C{}_{\geq1})\). Thus \(z\in Z_n(\kk Q_{\geq1})\). The converse is trivial. 
\end{proof}

Consequently, by Lemma~\ref{lem:faithful} and Lemma~\ref{lem:local},  $
z\in Z_n(\kk Q_{\geq 1}) $ if and only  if $
\Phi(z)\in Z_n(B_{H_C}(F_{H_C}))$.

Since $M_{H_C}(F_{H_C})$ is a full matrix algebra, we first
determine the $n$-centers of full matrix algebras.
For a nonisolated circle, we must also account for
matrix units outside $M_{H_C}(F_{H_C})$.

\begin{lemma}\label{lem:matrixkernel}
Over a characteristic-zero field $\kk$, for $r\geq1$ and $m\geq2$,
\[
 Z_m(\Mat_r(\kk))=\begin{cases}
 \kk I_r,&m<2r\text{ and }m\text{ is even},\\
 0,&m<2r\text{ and }m\text{ is odd},\\
 \Mat_r(\kk),&m\geq2r.
 \end{cases}
\]
\end{lemma}
\begin{proof}
The last case follows from Theorem~\ref{Thm:Amitsur-Levitski}. For $m<2r$, $\St_m$ is not a polynomial identity.

Let $U=Z_m(\Mat_r(\kk))$. By Lemma~\ref{lem:derivation}, $U$ is invariant under automorphisms and inner derivations. If $U$ contains a noncentral matrix, invariance under inner derivations implies that $U$ contains a matrix unit $E_{ij}$ for $i \neq j$.  For every invertible matrix \(P\in\Mat_r(\kk)\), the map $X \mapsto  PXP^{-1}$ is an automorphism. 
So, $U$ is invariant under conjugation by permutation matrices.  If \(E_{ij}\in U\) for some \(i\neq j\), conjugation by permutation matrices gives \(E_{kl}\in U\) for every \(k\neq l\). Since \(U\) is invariant under inner derivations,
$[E_{ij},E_{ji}]=E_{ii}-E_{jj}\in U$.
Hence \(\mathfrak{sl}_r(\kk)\subseteq U\).
Write \(X_j=X'_j+\lambda_jI_r\), where the trace of \(X'_j\) is zero.  Since \(\mathfrak{sl}_r(\kk)\subseteq U\), alternation implies that \(\St_m\) vanishes whenever one of its arguments lies in \(\mathfrak{sl}_r(\kk)\). By multilinearity,
$
\St_m(X_1,\ldots,X_m)
=\St_m(\lambda_1I_r,\ldots,\lambda_mI_r)=0$.
Hence \(\St_m\) vanishes  on \(\Mat_r(\kk)\), contradicting the fact that \(m<2r\). Therefore, $U$ contains no noncentral matrix, i.e.  \(U\subseteq \kk I_r\).

Finally,

\[
\begin{aligned}
\St_m(I_r,X_2,\cdots,X_m)
&=\sum_{\tau\in S_{\{2,\cdots,m\}}}
\operatorname{sgn}(\tau)
(\sum_{j=1}^m(-1)^{j-1})
X_{\tau(2)}\cdots X_{\tau(m)}\\
&=
\begin{cases}
0,& m\text{ even},\\
\St_{m-1}(X_2,\cdots,X_m),& m\text{ odd}.
\end{cases}
\end{aligned}
\] 
For odd $m<2r$, $\St_{m-1}$ is not a polynomial identity, so $Z_m(\Mat_r(\kk))=0$. 
\end{proof}
For an \(r\)-circle $C$, both \(L(C)\) and \(R(C)\) include the contribution \(w(C)=2r\) of \(C\). Put $
 \alpha(C)=L(C)-2r$ and $ \beta(C)=R(C)-2r$.
A matrix unit $E_{uv}\in B_{H_C}(F_{H_C})$ is called left-external if $C_u\preceq C_v\preceq C$ and $C_u\prec C$. It is called right-external if $C\preceq C_u\preceq C_v$ and $C\prec C_v$.

\begin{lemma}\label{lem:external}
Let $C$ be an $r$-circle with $L(C),R(C)<\infty$, and put $e_C=\sum_{i\in C_0}E_{ii}$.
\begin{enumerate}
\item Let $a\geq2$. If $a>\alpha(C)$, then for every tuple of left-external matrix units $x_1,\cdots,x_a$ one has $
 \St_a(x_1,\cdots,x_a)e_C=0$.  For every $2\leq a\leq\alpha(C)$ and every $i\in C_0$, the units can be chosen so that
$
 \St_a(x_1,\cdots,x_a)E_{ii}=\pm E_{ui}
$
for some vertex $u$.
\item Let $b\geq2$. If $b>\beta(C)$, then $e_C\St_b(z_1,\cdots,z_b)=0$ for every tuple of right-external units. For every $2\leq b\leq\beta(C)$ and every $j\in C_0$, such units can be chosen so that
$
 E_{jj}\St_b(z_1,\cdots,z_b)=\pm E_{jv}
$
for some vertex $v$.
\end{enumerate}
\end{lemma}
\begin{proof}
We prove the first statement.  Any nonzero  product of left-external units ending in $C$ follows a chain of SCCs
$
 D_1\prec\cdots\prec D_s\prec C$. Since \(L(C)<\infty\), every \(D_j\) is either trivial or a circle. If $D_j$ is trivial, there is no  matrix unit contained in $D_j$. 
Let  $D_1\prec\cdots\prec D_s\prec C$  be a chain arising from a nonzero  product, and let $T_j$ be the crossing unit from $D_j$ to the next SCC.  Let $Y_{j,1},\cdots,Y_{j,m_j}$ be the units lying entirely in $D_j$. Put $V_j=\St_{m_j}(Y_{j,1},\cdots,Y_{j,m_j})$ when $m_j\geq2$, and $V_j=Y_{j,1}$ when $m_j=1$. Summing the permutation terms with these  data, while varying only the order of the units contained in each $D_j$, gives, up to  sign, $V_1T_1\cdots V_sT_se_C$, where $V_j$ is omitted when $m_j=0$.
If $D_j$ is a $q_j$-circle and $m_j\geq2q_j$, then $\St_{m_j}(Y_{j,1},\cdots,Y_{j,m_j})=0$ by Theorem~\ref{Thm:Amitsur-Levitski}. Hence a nonzero sum of this form must satisfy $m_j\leq2q_j-1=w(D_j)-1$.

The chain can be extended to a path in the condensed quiver. There are at most $s$ crossing units. Therefore, if such a sum is nonzero, then
$
a\leq\sum_{j=1}^s(w(D_j)-1)+s
=\sum_{j=1}^s w(D_j)
\leq L(C)-2r=\alpha(C).
$
Thus, if $a>\alpha(C)$, every such sum is zero, and consequently
$\St_a(x_1,\cdots,x_a)e_C=0$.

For $\alpha(C)\geq2$, choose a chain of SCCs \(D_1\prec\cdots\prec D_s\prec C\) attaining \(L(C)\).
In each circle $D_j$, use the unique-order tuple
$
 E_{v_1v_1},E_{v_1v_2},E_{v_2v_2},\cdots,
 E_{v_{q_j-1}v_{q_j}},E_{v_{q_j}v_{q_j}}$,
use no  unit contained in a trivial SCC, join successive  blocks by corresponding matrix units, and finally use a corresponding matrix unit into the prescribed vertex $i\in C_0$. Exactly one permutation product is nonzero after  multiplication by $E_{ii}$. The tuple has $\alpha(C)$ entries. For each $2\leq a\leq\alpha(C)$, the last $a$ units have the same property and their nonzero product ends at $i$. The proof of the second statement is similar.
\end{proof}

We now combine the left-external matrix units, the elements
of $M_{H_C}(F_{H_C})$, and the right-external matrix units.
Their order in a nonzero product gives the following
factorization of the standard polynomial.

\begin{lemma}\label{lem:phase}
Let \(C\) be a circle component with \(L(C),R(C)<\infty\). Let $a,m,b\geq2$. Let $x_1,\cdots,x_a$ be left-external matrix units, let $y_1,\cdots,y_m\in M_{H_C}(F_{H_C})$, and let $z_1,\cdots,z_b$ be right-external matrix units. Then \[
\begin{aligned}
&\St_{a+m+b}(x_1,\cdots,x_a,y_1,\cdots,y_m,z_1,\cdots,z_b)\\
&\qquad=\St_a(x_1,\cdots,x_a)\,
  \St_m(y_1,\cdots,y_m)\,
  \St_b(z_1,\cdots,z_b).
\end{aligned}
\]
The same factorization holds when any of the three blocks has length zero or one. Indeed, a block of length zero is omitted, while a block consisting of a single element is replaced by that element.

More generally, in any nonzero permutation product of matrix units from $B_{H_C}(F_{H_C})$ containing a unit $E_{ij}$ with $i,j\in C_0$, every left-external unit occurs before every unit $E_{kl}$ with $k,l\in C_0$, every right-external unit occurs after every such unit, and no unit $E_{xy}$ with $C_x\prec C\prec C_y$ can occur.

\end{lemma}

\begin{proof}
We first prove the ordering assertion. Consider a nonzero product
$
E_{u_1v_1}\cdots E_{u_sv_s}
$
of matrix units from \(B_{H_C}(F_{H_C})\), and suppose that one of these units is \(E_{ij}\) with \(i,j\in C_0\). Since the product is nonzero, \(v_\nu=u_{\nu+1}\) for every \(\nu\). Moreover, \(E_{uv}\in B_{H_C}(F_{H_C})\) implies \(C_u\preceq C_v\). Hence $C_{u_1}\preceq C_{v_1}=C_{u_2}\preceq C_{v_2}
=\cdots=C_{u_s}\preceq C_{v_s}$.

Suppose that a left-external unit \(E_{uv}\) occurs after a matrix unit \(E_{ij}\) with \(i,j\in C_0\). The part of the product between these two units gives \(C\preceq C_u\), contradicting \(C_u\prec C\). Therefore every left-external unit occurs before every unit contained in \(C\).
Dually, every right-external unit occurs after every unit contained in \(C\).
Moreover, a unit \(E_{xy}\) with \(C_x\prec C\prec C_y\) can occur neither before nor after a unit contained in \(C\). Thus no such matrix unit can occur.

By multilinearity, it is enough to assume that \(y_1,\cdots,y_m\) are matrix units in \(M_{H_C}(F_{H_C})\). By the ordering assertion,  a permutation can give a nonzero product only if the \(x_i\)'s occur before the \(y_j\)'s and the \(z_k\)'s occur after the \(y_j\)'s. Thus such a permutation only permutes the elements within the three blocks. The sign of such a permutation is the product of the signs of the three corresponding permutations. Hence

\[
\begin{aligned}
&\St_{a+m+b}(x_1,\cdots,x_a,y_1,\cdots,y_m,z_1,\cdots,z_b)\\
&\qquad=\St_a(x_1,\cdots,x_a)\,
\St_m(y_1,\cdots,y_m)\,
\St_b(z_1,\cdots,z_b).
\end{aligned}
\]

\end{proof}

Combining this factorization with
Lemma~\ref{lem:matrixkernel} and Lemma~\ref{lem:external},
we determine the contribution of a nonisolated circle
component to $Z_n(\kk Q_{\geq1})$.

\begin{proposition}\label{prop:circlethreshold}
Let $C$ be a nonisolated  $r$-circle with $L(C),R(C)<\infty$. Then
$$
 Z_n(\kk Q_{\geq 1})\cap A_{CC}=\begin{cases}
 0,&n<h(C),\\
 A_{CC},&n\geq h(C),
 \end{cases}
$$
where $ h(C)=L(C)+R(C)-2r$.
\end{proposition}
\begin{proof}
By Lemma~\ref{lem:faithful} and Lemma~\ref{lem:local}, it suffices to work in $B_{H_C}(F_{H_C})$. Write $
 \alpha=L(C)-2r $ and $\beta=R(C)-2r$,
so that $h(C)=2r+\alpha+\beta$. Since $C$ is nonisolated, $\alpha+\beta\geq1$.

Assume first that \(n\geq h(C)\). Since \(Z_n(B_{H_C}(F_{H_C}))\) is linear, it suffices to consider a matrix unit \(E_{ij}\in M_{H_C}(F_{H_C})\). 
By multilinearity, it suffices to take the other \(n-1\) elements to be matrix units of \(B_{H_C}(F_{H_C})\). 
If every permutation product is zero, there is nothing to prove.
Otherwise, by Lemma~\ref{lem:phase}, after reordering these matrix units and changing only an overall sign, the standard polynomial factors into a left block of length \(a\), a middle block of length \(m\) containing \(E_{ij}\), and a right block of length \(b\), where \(a+m+b=n\).
A nonempty left block implies \(\alpha\ge1\), and a nonempty right block implies \(\beta\ge1\).
If \(a>\alpha\), then $a\geq2$ and \(\St_a(x_1,\cdots,x_a)e_C=0\) by Lemma~\ref{lem:external}, so the whole product is zero. Similarly, if \(b>\beta\), the whole product is zero. It remains to consider the case where \(a\leq\alpha\) and \(b\leq\beta\). In this case,
$
m=n-a-b\geq n-\alpha-\beta\geq 2r
$.
Hence the middle factor vanishes by Theorem~\ref{Thm:Amitsur-Levitski}. Thus
\(\St_n(E_{ij},X_2,\cdots,X_n)=0\)
for all \(X_2,\cdots,X_n\in B_{H_C}(F_{H_C})\). Therefore \(E_{ij}\in Z_n(B_{H_C}(F_{H_C}))\). Hence \(M_{H_C}(F_{H_C})\subseteq Z_n(B_{H_C}(F_{H_C}))\), and so \(A_{CC}\subseteq Z_n(\kk Q_{\geq 1})\).

Now assume \(n<h(C)\), and let \(0\ne Z\in M_{H_C}(F_{H_C})\). We shall choose \(n-1\) further elements for which the corresponding  value of $\St_n$ is nonzero.
Choose a positive odd integer \(m<2r\) and nonnegative integers \(a\leq\alpha\), \(b\leq\beta\) such that \(a+m+b=n\). Indeed,  if \(n<2r\) is odd, take \(m=n\) and \(a=b=0\). If \(n<2r\) is even, take \(m=n-1\). Since \(\alpha+\beta\geq1\), one may take \(a+b=1\). If \(n\geq2r\), take \(m=2r-1\). Then \(n-m=n-2r+1\leq\alpha+\beta\), since \(n<2r+\alpha+\beta\), so \(a\) and \(b\) can be chosen with \(a\leq\alpha\), \(b\leq\beta\), and \(a+b=n-m\).

Put $W=Z$ if $m=1$. If $m\geq3$, Lemma~\ref{lem:matrixkernel} gives $Y_2,\cdots,Y_m\in M_{H_C}(F_{H_C})$ such that $W=\St_m(Z,Y_2,\cdots,Y_m)\ne0$. Choose \(i,j\in C_0\) with \(W_{ij}\ne0\).

If $a\geq2$, choose left-external matrix units
$x_1,\cdots,x_a$ using the construction in the proof of
Lemma~\ref{lem:external}, with the last crossing unit
ending at $i$.
Every nonzero  product of these units ends at $i$.
Hence
$\St_a(x_1,\cdots,x_a)e_C
=\St_a(x_1,\cdots,x_a)E_{ii}
=\pm E_{ui}$
for some vertex $u$.
If $a=1$, choose a left-external unit $x_1=E_{ui}$.
Similarly, if $b\geq2$, use the construction in the proof of
Lemma~\ref{lem:external} to choose right-external matrix units
$z_1,\cdots,z_b$ such that every nonzero product
starts at $j$. Then
$e_C\St_b(z_1,\cdots,z_b)
=E_{jj}\St_b(z_1,\cdots,z_b)
=\pm E_{jv}$
for some vertex $v$.
If $b=1$, choose a right-external unit $z_1=E_{jv}$.
Since $W=e_CWe_C$, Lemma~\ref{lem:phase} shows that
the corresponding value of $\St_n$ is nonzero.
If $a,b>0$, it is, up to sign,
$E_{ui}WE_{jv}=W_{ij}E_{uv}\ne0$.
If $a=0<b$, it is, up to sign, $WE_{jv}\ne0$.
If $b=0<a$, it is, up to sign, $E_{ui}W\ne0$.
If $a=b=0$, it is $W\ne0$.
Thus \(Z\notin Z_n(B_{H_C}(F_{H_C}))\). Since \(Z\ne0\) was arbitrary, \(Z_n(B_{H_C}(F_{H_C}))\cap M_{H_C}(F_{H_C})=0\). Lemma~\ref{lem:local} gives $Z_n(\kk Q_{\geq 1})\cap A_{CC}=0$.

\end{proof}

For an isolated circle $C$, there are no left-external
or right-external matrix units, so
$B_C(F_C)=M_C(F_C)$.
The distinction between even and odd degrees in
Lemma~\ref{lem:matrixkernel} therefore appears directly
in the description of $Z_n(\kk C_{\geq1})$.

\begin{proposition}\label{prop:isolated}
Let $C$ be an isolated $r$-circle. Then
$$
 Z_n(\kk C_{\geq1})=\begin{cases}
 \displaystyle\bigoplus_{k\geq1}\kk(c_1^k+\cdots+c_r^k),
       &n<2r\text{ and }n\text{ even},\\[2mm]
 0,&n<2r\text{ and }n\text{ odd},\\
 \kk C_{\geq1},&n\geq2r.
 \end{cases}
$$
\end{proposition}
\begin{proof}
Here $H_C=C$ and Lemma~\ref{lem:faithful} identifies the $F_C$-span of the image with $\Mat_r(F_C)$. By Lemma~\ref{lem:matrixkernel}, 
when \(n<2r\) is even,
the image of \(Z_n(\kk C_{\geq1})\) lies in $F_C I_r$. For each $k$, the paths $c_i^k$ all map to the same nonzero  monomial in $F_C$ times $E_{ii}$, so their coefficients must agree. Moreover, $\Phi(\sum_{i=1}^r c_i^k)$ is a multiple of $I_r$.  This gives exactly the first case. The other two cases are immediate from  Lemma~\ref{lem:matrixkernel}.
\end{proof}

\subsection{Completion of the proof}\label{sub:completion}

\begin{proof}[Proof of Theorem~\ref{thm:main}]
Let $z\in Z_n(\kk Q_{\geq 1})$. Project $z$ onto each  $e_i\kk Q_{\geq 1}e_j$ where $i\ne j$ by inner derivations. If $i,j$ lie in distinct SCCs, Proposition~\ref{prop:cross} gives exactly the first summand of $\mathcal{I}_n$. If they lie in the same noncircle cyclic component, Lemma~\ref{lem:bad} gives zero. If they lie in a circle component $C$ comparable with a noncircle cyclic component, equivalently, at least one of $L(C),R(C)$ is infinite,  Lemma~\ref{lem:comparable} gives zero. Every remaining nonisolated circle is treated by Proposition~\ref{prop:circlethreshold}, and every isolated circle by Proposition~\ref{prop:isolated}.

Now consider the remaining closed paths in $z$. Lemma~\ref{lem:bad} removes those in noncircle cyclic components, and Lemma~\ref{lem:comparable} removes those in circles comparable with a noncircle cyclic component. For any remaining circle component $C$, arrow-degree decomposition reduces the remaining part, for each fixed $k\geq1$, to
$
\sum_{i\in C_0}\lambda_i c_i^k$.
If $C$ is nonisolated, Proposition~\ref{prop:circlethreshold} shows that, for each fixed $k$, the coefficients $\lambda_i$ are arbitrary when $h(C)\leq n$, and all vanish when $h(C)>n$. If $C$ is isolated, Proposition~\ref{prop:isolated} shows that, for each fixed $k$, the coefficients are arbitrary when $n\geq2r_C$, all vanish when $n<2r_C$ is odd, and are all equal when $n<2r_C$ is even. Thus in the last case the contribution for this fixed $k$ is
$
\kk\sum_{i\in C_0}c_i^k$.

Thus all terms in \eqref{eq:main} have been obtained. The converse inclusion follows from Proposition~\ref{prop:cross}, Proposition~\ref{prop:circlethreshold}, and Proposition~\ref{prop:isolated}.

\end{proof}

\section{Consequences and examples}\label{sec:consequences}
In this section, we consider some consequences and examples of the main theorem.

\begin{proposition}\label{prop:ideal}
For every \(n\geq2\), the subspace \(\mathcal{I}_n\) is a two-sided ideal of \(\kk Q_{\geq1}\), and
$
 \mathcal{I}_n\subseteq \mathcal{I}_{n+1}$.
If \(Q\) has no isolated \(r\)-circle with \(r\geq2\), then
$
 Z_2(\kk Q_{\geq1})\subseteq Z_3(\kk Q_{\geq1})
 \subseteq Z_4(\kk Q_{\geq1})\subseteq\cdots,
$
and every \(Z_n(\kk Q_{\geq1})\) is a two-sided ideal. For an arbitrary finite quiver \(Q\), every \(Z_n(\kk Q_{\geq1})\) is a subalgebra of \(\kk Q_{\geq1}\).
\end{proposition}

\begin{proof}
The inclusion \(\mathcal{I}_n\subseteq \mathcal{I}_{n+1}\) follows directly from the definition of \(\mathcal{I}_n\).

We show that \(\mathcal{I}_n\) is a two-sided ideal. First let \(p\in A_{CD}\), where \(C\ne D\) and \(L(C)+R(D)\leq n\). Suppose \(qp\ne0\), and let \(E\) be the SCC containing \(s(q)\). Then \(E\preceq C\), so \(L(E)\leq L(C)\). Hence
\(L(E)+R(D)\leq n\), and therefore \(qp\in \mathcal{I}_n\). The proof for \(pq\) is similar.

Now let \(C\) be an \(r\)-circle with \(h(C)\leq n\), and let \(p\in A_{CC}\). Suppose \(qp\ne0\). If \(s(q)\in C_0\), then \(qp\in A_{CC}\). Otherwise, let \(E\prec C\) be the SCC containing \(s(q)\). A chain attaining \(L(E)\) can be extended to \(C\), so
\(L(E)+2r\leq L(C)\). Thus
$
 L(E)+R(C)\leq L(C)+R(C)-2r=h(C)\leq n$,
and \(qp\in \mathcal{I}_n\). The proof for \(pq\) is similar. Hence \(\mathcal{I}_n\) is a two-sided ideal.

If \(Q\) has no isolated \(r\)-circle with \(r\geq2\), the additional summand in Theorem~\ref{thm:main} is empty for every \(n\geq2\). Thus \(Z_n(\kk Q_{\geq1})=\mathcal{I}_n\), and the second assertion follows.

For arbitrary \(Q\), let \(C\) be an isolated \(r\)-circle which occurs in the additional summand of Theorem~\ref{thm:main}, and put
\(s_{C,k}=\sum_{i\in C_0}c_i^k\). Then
\(s_{C,k}s_{C,l}=s_{C,k+l}\).
Products between different isolated SCCs are zero. Moreover, in this case \(n<2r=h(C)\), so \(A_{CC}\) does not occur in \(\mathcal{I}_n\). Since \(C\) is isolated, the products of \(s_{C,k}\) with \(\mathcal{I}_n\) are also zero. Theorem~\ref{thm:main} therefore shows that \(Z_n(\kk Q_{\geq1})\) is closed under multiplication.
\end{proof}

\begin{proposition}\label{prop:stabilization}
Let \(N=|Q_0|\). 
Then
$
 Z_n(\kk Q_{\geq1})=Z_{2N}(\kk Q_{\geq1})$
for every $n\geq2N$.
Moreover, for \(n\geq2N\),

\[
 Z_n(\kk Q_{\geq1})
 =
 \Span_\kk\left\{
 p \;\middle|\;
 \begin{array}{l}
 p\text{ is a non-trivial path such that}\\
 \text{there is no noncircle cyclic SCC }D
 \text{ with }D\preceq C_{s(p)},\\
 \text{and there is no noncircle cyclic SCC }D
 \text{ with }C_{t(p)}\preceq D
 \end{array}
 \right\}.
\]
\end{proposition}

\begin{proof}

If \(D\) is an SCC with \(w(D)<\infty\),  then \(D\) is trivial or a circle, and
\(w(D)\leq2|D_0|\). Let $\mathcal T$ be the subset of $\mathcal S(Q)$  consisting of all SCCs with $w(D)<\infty$. Then we have  $\sum_{D\in\mathcal T} w(D)\leq 2N$.

Let \(C\prec D\), with \(L(C),R(D)<\infty\). Choose chains attaining \(L(C)\) and \(R(D)\). These two chains have no common SCC. Indeed, a common SCC, together with \(C\prec D\), would give an oriented cycle in the condensed quiver. Therefore
\(L(C)+R(D)\leq2N\).

Now let \(C\) be an \(r\)-circle with \(L(C),R(C)<\infty\). A chain attaining \(L(C)\) and a chain attaining \(R(C)\) can meet only at \(C\). Since \(w(C)=2r\),
$h(C)=L(C)+R(C)-2r\leq2N$.
Also \(2r\leq2N\). 

Therefore, if $n\geq2N$, every pair $C\ne D$ with
$L(C),R(D)<\infty$ and $A_{CD} \neq 0$ satisfies
$L(C)+R(D)\leq2N\leq n$, and every circle $C$ with
$L(C),R(C)<\infty$ satisfies
$h(C)\leq2N\leq n$.
Moreover, for an isolated $r$-circle, $2r\leq2N$, so the case
$n<2r$ does not occur. Hence
$
 Z_n(\kk Q_{\geq1})=Z_{2N}(\kk Q_{\geq1})$.

 It remains to identify $Z_{2N}(\kk Q_{\geq1})$.
By Theorem~\ref{thm:main}, since $2r_C\leq2N$ for every
circle component $C$, we have
$Z_{2N}(\kk Q_{\geq1})=\mathcal I_{2N}$.
Thus this space is spanned by paths.
Let $p$ be a non-trivial path in $Q$.

For every SCC $C$, $L(C)<\infty$ if and only if no
noncircle cyclic SCC can reach $C$, and $R(C)<\infty$
if and only if no noncircle cyclic SCC can be reached
from $C$.

Suppose first that $C_{s(p)}\ne C_{t(p)}$.
Since $p$ joins these two components,
$C_{s(p)}\prec C_{t(p)}$.
By Theorem~\ref{thm:main} and the bound proved above,
$p\in Z_{2N}(\kk Q_{\geq1})$ if and only if
$L(C_{s(p)}),R(C_{t(p)})<\infty$.
This is equivalent to the two  conditions
in the statement.
Now suppose that $C_{s(p)}=C_{t(p)}=C$.
Since $p$ is non-trivial, $C$ is cyclic.
If the two  conditions in the statement hold, then
$L(C),R(C)<\infty$.
In particular, $C$ is not a noncircle cyclic component,
so $C$ is a circle.
The bound proved above gives $h(C)\leq2N$, and hence
$p\in Z_{2N}(\kk Q_{\geq1})$ by
Theorem~\ref{thm:main}.

Conversely, if $p\in Z_{2N}(\kk Q_{\geq1})$, then
Theorem~\ref{thm:main} shows that $C$ is a circle
with $h(C)\leq2N$.
Thus $L(C),R(C)<\infty$, so the two 
conditions in the statement hold.

Therefore the paths belonging to
$Z_{2N}(\kk Q_{\geq1})$ are exactly those listed
in the statement. Since this space is spanned by
paths, the statement follows.

\end{proof}

We briefly explain how the preceding results apply to algebras of formal languages. See \cite{MR4412543},  \cite{MR1475463} for background.

Let \(\mathcal A\) be a finite alphabet and let \(\mathcal A^+\) be the set of all nonempty finite words over \(\mathcal A\). For a word $\omega=a_1\cdots a_m$ and
$1\leq i\leq j\leq m$, write
$\omega[i,j]=a_i\cdots a_j$.
Here a subword means a sequence of consecutive letters
in a word. Let \(\mathcal W\) be a finite nonempty set of forbidden words in \(\mathcal A^+\). Define
$$
 X_{\mathcal W}
 =
 \{\omega\in\mathcal A^+ \mid 
 \omega\text{ contains no word of }\mathcal W\text{ as a subword}\}.$$
For \(m\geq1\), let \(B_m(X_{\mathcal W})\) denote the set of words of length \(m\) in \(X_{\mathcal W}\).
Let \(k\) be the maximal length of the words in \(\mathcal W\). If \(k\geq2\), regard \(B_k(X_{\mathcal W})\) as a new alphabet. The \(k\)-th block formal language \(X_{\mathcal W}^{[k]}\) consists of all words
\(\omega_1\cdots\omega_m\), with \(\omega_i\in B_k(X_{\mathcal W})\), satisfying
$
 \omega_i[2,k]=\omega_{i+1}[1,k-1]
$
for every \(1\leq i<m\).

Let \(\kk X_{\mathcal W}\) and \(\kk X_{\mathcal W}^{[k]}\) be the \(\kk\)-vector spaces with bases \(X_{\mathcal W}\) and \(X_{\mathcal W}^{[k]}\), respectively, with multiplication given by concatenation when the concatenated word remains in the corresponding language, and zero otherwise.

\begin{proposition}\label{prop:languagepath}
Let \(\mathcal A\), \(\mathcal W\), and \(k\) be as above.
If \(k=1\), there exists a finite quiver \(Q_{\mathcal W}\) such that
$
 \kk X_{\mathcal W}\cong \kk Q_{\mathcal W}{}_{\geq1}$.
If \(k\geq2\),  there exists a finite quiver \(Q_{\mathcal W}\) such that 
$
  \kk X_{\mathcal W}^{[k]}
 \cong
 \kk Q_{\mathcal W}{}_{\geq1}$.

\end{proposition}

\begin{proof}
Suppose first that \(k=1\). Let \(Q_{\mathcal W}\) have one vertex and one loop for each letter in \(B_1(X_{\mathcal W})\). It is easy to check that 
$
 \kk X_{\mathcal W}\cong \kk Q_{\mathcal W}{}_{\geq1}$.

Now suppose that \(k\geq2\). Define \(Q_{\mathcal W}\) by
$
 (Q_{\mathcal W})_0=B_{k-1}(X_{\mathcal W})$,
 $(Q_{\mathcal W})_1=B_k(X_{\mathcal W})$,
and, for \(\omega\in B_k(X_{\mathcal W})\), set
\(s(\omega)=\omega[1,k-1]\) and \(t(\omega)=\omega[2,k]\).
One can check that
$
 \kk X_{\mathcal W}^{[k]}
 \cong
 \kk Q_{\mathcal W}{}_{\geq1}$.

\end{proof}

Consequently, for every \(n\geq2\),
$
 Z_n(\kk X_{\mathcal W})
 \cong
 Z_n(\kk Q_{\mathcal W}{}_{\geq1})
$
when \(k=1\), and
$
 Z_n(\kk X_{\mathcal W}^{[k]})
 \cong
 Z_n(\kk Q_{\mathcal W}{}_{\geq1})
$
when \(k\geq2\). Therefore all the preceding results on \(Z_n\) for principal subalgebras apply directly to algebras of block formal languages. We do not rewrite the resulting description in the language of words, since it follows immediately from these isomorphisms.

For a graded subspace \(V\subseteq \kk Q_{\geq1}\), let
$
H_V(t)=\sum_{m\geq1}\dim_{\kk}V_m\,t^m$
denote its Hilbert series, where
\(V_m=V\cap \operatorname{span}_{\kk}\{p \mid |p|=m\}\).

We use the standard terminology and basic results on regular languages as in
\cite{MR3556290}. For a language \(\mathcal L\) over a finite
alphabet, its Hilbert series is
$
H_{\mathcal L}(t)
=
\sum_{w\in\mathcal L}t^{|w|}
=
\sum_{m\geq0}
\#\{w\in\mathcal L \mid |w|=m\}t^m$.

\begin{proposition}\label{prop:hilbert}
For every finite quiver \(Q\) and every \(n\geq2\), let
\(\mathcal L_n\) be the set of paths whose \(\kk\)-span is \(\mathcal{I}_n\).
Then \(\mathcal L_n\) is a regular language over the alphabet \(Q_1\).
Put \(N=|Q_0|\), let \(M\) be the adjacency matrix of \(Q\), and define

$$
\begin{aligned}
 \Omega_n={}&
 \{(i,j)\in Q_0^2 \mid 
 C_i\ne C_j,\ L(C_i)+R(C_j)\leq n\}\\
 &\cup
 \{(i,j)\in Q_0^2 \mid 
 C_i=C_j=C,\ C\text{ is a circle},\ h(C)\leq n\}.
\end{aligned}$$
Then
$
 H_{\mathcal L_n}(t)
 =
 H_{\mathcal{I}_n}(t)
 =
 \sum_{(i,j)\in\Omega_n}
 \bigl[tM(I_N-tM)^{-1}\bigr]_{ij}$.
Moreover, 
\[
H_{Z_n(\kk Q_{\geq1})}(t)
=
\begin{cases}
H_{\mathcal L_n}(t),
& n\text{ is odd},\\[8pt]
\displaystyle
H_{\mathcal L_n}(t)
+
\sum_{\substack{
C\text{ isolated }r_C\text{-circle}\\
2r_C>n}}
\frac{t^{r_C}}{1-t^{r_C}},
& n\text{ is even}.
\end{cases}
\]
In particular, \(H_{Z_n(\kk Q_{\geq1})}(t)\) is a rational function.
\end{proposition}

\begin{proof}
For fixed vertices \(i,j\), the set of non-trivial paths from \(i\) to
\(j\) is a regular language, since it is recognized by a finite automaton which
records the current vertex and whether at least one arrow has been read.
By the definition of \(\mathcal{I}_n\),
$
 \mathcal L_n
 =
 \bigcup_{(i,j)\in\Omega_n}
 \{p \mid s(p)=i,\ t(p)=j,\ |p|\geq1\}$.
Since \(\Omega_n\) is finite and regular languages are closed under finite
unions, \(\mathcal L_n\) is regular.

For every
\(m\geq1\),
$
 \dim_{\kk}
 (\mathcal{I}_n\cap\operatorname{span}_{\kk}\{p \mid |p|=m\})
 =
 \#\{p\in\mathcal L_n \mid |p|=m\}$.
Therefore
\(H_{\mathcal{I}_n}(t)=H_{\mathcal L_n}(t)\).
For every \(m\geq1\), the entry \((M^m)_{ij}\) is the number of paths of
length \(m\) from \(i\) to \(j\). Thus
$
 \sum_{m\geq1}M^mt^m
 =
 tM(I_N-tM)^{-1}$,
and consequently
$
 H_{\mathcal L_n}(t)
 =
 \sum_{(i,j)\in\Omega_n}
 \bigl[tM(I_N-tM)^{-1}\bigr]_{ij}$.

By Theorem~\ref{thm:main}, the remaining summands of
\(Z_n(\kk Q_{\geq1})\) occur precisely for isolated \(r_C\)-circles with
\(n\) even and \(2r_C>n\). For each such circle \(C\), the corresponding
subspace has the basis
$
 \left\{
 \sum_{i\in C_0}c_i^k \mid k\geq1
 \right\}$.
Since \(\sum_{i\in C_0}c_i^k\) has degree \(kr_C\), the Hilbert series of the corresponding subspace is \(\sum_{k\ge1}t^{kr_C}=\frac{t^{r_C}}{1-t^{r_C}}\).
The formula for \(H_{Z_n(\kk Q_{\geq1})}(t)\) now follows from
Theorem~\ref{thm:main}. Hence
\(H_{Z_n(\kk Q_{\geq1})}(t)\) is rational.
\end{proof}

We conclude with some examples.
\begin{example}\label{cor:acyclic}
Suppose $Q$ is acyclic. Let $d_-(i)$ be the maximum length of a path ending at $i$, and let $d_+(j)$ be the maximum length of a path starting at $j$, where paths of length zero are allowed. Then  $
 Z_n(\kk Q_{\geq1})=\Span_\kk\{p \mid p \text{ is non-trivial},\  d_-(s(p))+d_+(t(p))\leq n-2\}$.
If $d$ is the maximum length of all paths and $\kk Q_{\geq1}\ne0$, the least $n\geq2$ for which $Z_n(\kk Q_{\geq1})=\kk Q_{\geq1}$ is $d+1$. Equivalently, $d+1$ is the least integer $n\ge2$ such that
$\St_n$ is a polynomial identity of $\kk Q_{\ge1}$.
\end{example}
\begin{proof}
All SCCs are trivial, with $L(C_i)=d_-(i)+1$ and $R(C_j)=d_+(j)+1$. Theorem~\ref{thm:main} gives the formula. For every non-trivial path $p$, concatenating a longest incoming path, $p$, and a longest outgoing path gives
$
 d_-(s(p))+|p|+d_+(t(p))\leq d$.
Thus $n=d+1$ suffices. An arrow $a$ on a path of length $d$ has $d_-(s(a))+d_+(t(a))=d-1$, so no smaller $n\geq2$ works when $d\geq2$. The case $d=1$ is trivial.
\end{proof}

\begin{example}\label{cor:circle}
For an isolated  $r$-circle $C_r$,
\[
 Z_n(\kk C_r{}_{\geq1})=\begin{cases}
 \displaystyle\bigoplus_{k\geq1}\kk(c_1^k+\cdots+c_r^k),
       &n<2r\text{ and }n\text{ even},\\[2mm]
 0,&n<2r\text{ and }n\text{ odd},\\
 \kk C_r{}_{\geq1},&n\geq2r.
 \end{cases}
\]
\end{example}

\begin{example}\label{cor:center}
The center of $\kk Q_{\geq1}$ is the direct sum of the $\kk$-span of non-trivial paths $p$ with $s(p)$ a source and $t(p)$ a sink, and the spaces
$
 \bigoplus_{k\geq1}\kk(c_1^k+\cdots+c_r^k)$
for isolated circle components.
\end{example}
\begin{proof}
For paths joining distinct SCCs, $L+R\leq2$ forces $L=R=1$,  
which means precisely that such paths start at sources and end at sinks.
A nonisolated circle has $h(C)>2r\geq2$ and contributes nothing to $Z_2$. The isolated-circle contribution is Example~\ref{cor:circle}.
\end{proof}

\begin{example}\label{thm:third}
For every finite quiver $Q$,
\begin{align*}
Z_3(\kk Q_{\geq1})
={}&\Span_\kk\{p \mid p\text{ is non-trivial},\
C_{s(p)}\ne C_{t(p)},\
L(C_{s(p)})+R(C_{t(p)})\leq3\}\\
&\oplus
\bigoplus_{\substack{
C\in\mathcal S(Q)\\
C\text{ is a }1\text{-circle}\\
L(C)+R(C)-2\leq3}}
\Span_\kk\{c^k\mid k\geq1,\ c\text{ is the unique
loop of } C \}.
\end{align*}

\end{example}
\begin{proof}
For paths whose starting and terminal vertices lie in distinct SCCs, the result follows from Proposition~\ref{prop:cross}. If \(C\) is a noncircle cyclic component, Lemma~\ref{lem:bad} gives no contribution. Let $C$ be a circle of size $r\geq2$. Suppose that  $C$ is isolated. Since  $3<2r$ and $3$ is odd,  Example~\ref{cor:circle} gives no contribution.
Now suppose that $C$ is nonisolated.
If $L(C)=\infty$ or $R(C)=\infty$, then $C$ is comparable
with a noncircle cyclic component, so Lemma~\ref{lem:comparable}
gives no contribution.
Otherwise, $L(C),R(C)<\infty$ and $h(C)\geq2r+1\geq5$,
so Proposition~\ref{prop:circlethreshold} gives no contribution.
Let $C$ be a $1$-circle. Theorem~\ref{thm:main} gives the stated condition $h(C)=L(C)+R(C)-2\leq3$.
\end{proof}

\noindent\textbf{AI Declaration.}
The proof of Theorem~\ref{thm:main} was found by  ChatGPT 6. The authors' contribution to this proof is to verify, organize, and present it in a rigorous form for a human audience. The manuscript itself was written by the authors. ChatGPT 6 was additionally used for language refinement, proofreading, and checking the presentation of some parts of the manuscript. The authors have independently verified the mathematical content and take full responsibility for the exposition and correctness of the manuscript.

\bibliographystyle{plain}
\bibliography{Zn}

@article {MR36751,
    AUTHOR = {Amitsur, A. S. and Levitzki, J.},
     TITLE = {Minimal identities for algebras},
   JOURNAL = {Proc. Amer. Math. Soc.},
  FJOURNAL = {Proceedings of the American Mathematical Society},
    VOLUME = {1},
      YEAR = {1950},
     PAGES = {449--463},
      ISSN = {0002-9939,1088-6826},
   MRCLASS = {09.1X},
  MRNUMBER = {36751},
MRREVIEWER = {I.\ Kaplansky},
       DOI = {10.2307/2032312},
       URL = {https://doi.org/10.2307/2032312},
}

@book {MR4249615,
    AUTHOR = {Aljadeff, Eli and Giambruno, Antonio and Procesi, Claudio and
              Regev, Amitai},
     TITLE = {Rings with polynomial identities and finite dimensional
              representations of algebras},
    SERIES = {American Mathematical Society Colloquium Publications},
    VOLUME = {66},
 PUBLISHER = {American Mathematical Society, [Providence], RI},
      YEAR = {[2020] \copyright 2020},
     PAGES = {xii+630},
      ISBN = {978-1-4704-5174-5},
   MRCLASS = {16R20 (14L24 15A72 16H05 16H10 16W22)},
  MRNUMBER = {4249615},
MRREVIEWER = {Daniela\ La Mattina},
}

@book {MR2176105,
    AUTHOR = {Giambruno, Antonio and Zaicev, Mikhail},
     TITLE = {Polynomial identities and asymptotic methods},
    SERIES = {Mathematical Surveys and Monographs},
    VOLUME = {122},
 PUBLISHER = {American Mathematical Society, Providence, RI},
      YEAR = {2005},
     PAGES = {xiv+352},
      ISBN = {0-8218-3829-6},
   MRCLASS = {16R10 (16P90 20C30)},
  MRNUMBER = {2176105},
MRREVIEWER = {E.\ Formanek},
       DOI = {10.1090/surv/122},
       URL = {https://doi.org/10.1090/surv/122},
}

@book {MR3308668,
    AUTHOR = {Schiffler, Ralf},
     TITLE = {Quiver representations},
    SERIES = {CMS Books in Mathematics/Ouvrages de Math\'ematiques de la
              SMC},
 PUBLISHER = {Springer, Cham},
      YEAR = {2014},
     PAGES = {xii+230},
      ISBN = {978-3-319-09203-4; 978-3-319-09204-1},
   MRCLASS = {16G20 (16G70)},
  MRNUMBER = {3308668},
MRREVIEWER = {Alex\ Martsinkovsky},
       DOI = {10.1007/978-3-319-09204-1},
       URL = {https://doi.org/10.1007/978-3-319-09204-1},
}

@book {MR2197389,
    AUTHOR = {Assem, Ibrahim and Simson, Daniel and Skowro\'nski, Andrzej},
     TITLE = {Elements of the representation theory of associative algebras.
              {V}ol. 1},
    SERIES = {London Mathematical Society Student Texts},
    VOLUME = {65},
      NOTE = {Techniques of representation theory},
 PUBLISHER = {Cambridge University Press, Cambridge},
      YEAR = {2006},
     PAGES = {x+458},
      ISBN = {978-0-521-58423-4; 978-0-521-58631-3; 0-521-58631-3},
   MRCLASS = {16G10 (16-02)},
  MRNUMBER = {2197389},
MRREVIEWER = {Peter\ W.\ Donovan},
       DOI = {10.1017/CBO9780511614309},
       URL = {https://doi.org/10.1017/CBO9780511614309},
}

@book {MR3727119,
    AUTHOR = {Derksen, Harm and Weyman, Jerzy},
     TITLE = {An introduction to quiver representations},
    SERIES = {Graduate Studies in Mathematics},
    VOLUME = {184},
 PUBLISHER = {American Mathematical Society, Providence, RI},
      YEAR = {2017},
     PAGES = {x+334},
      ISBN = {978-1-4704-2556-2},
   MRCLASS = {16-02 (13A50 14L24 16G10 16G20 16G70)},
  MRNUMBER = {3727119},
MRREVIEWER = {Xueqing\ Chen},
       DOI = {10.1090/gsm/184},
       URL = {https://doi.org/10.1090/gsm/184},
}

@article {MR3712583,
    AUTHOR = {Corrales Garc\'ia, Mar\'ia G. and Mart\'in Barquero, Dolores
              and Mart\'in Gonz\'alez, C\'andido and Siles Molina, Mercedes
              and Solanilla Hern\'andez, Jos\'e{} F.},
     TITLE = {Centers of path algebras, {C}ohn and {L}eavitt path algebras},
   JOURNAL = {Bull. Malays. Math. Sci. Soc.},
  FJOURNAL = {Bulletin of the Malaysian Mathematical Sciences Society},
    VOLUME = {40},
      YEAR = {2017},
    NUMBER = {4},
     PAGES = {1745--1767},
      ISSN = {0126-6705,2180-4206},
   MRCLASS = {46L55 (05C25 16D70)},
  MRNUMBER = {3712583},
       DOI = {10.1007/s40840-015-0214-1},
       URL = {https://doi.org/10.1007/s40840-015-0214-1},
}

@book {MR4412543,
    AUTHOR = {Lind, Douglas and Marcus, Brian},
     TITLE = {An introduction to symbolic dynamics and coding},
    SERIES = {Cambridge Mathematical Library},
   EDITION = {Second},
 PUBLISHER = {Cambridge University Press, Cambridge},
      YEAR = {2021},
     PAGES = {xix+550},
      ISBN = {978-1-108-82028-8},
   MRCLASS = {37-02 (37B10 37B40 37B51 37D35 94Bxx)},
  MRNUMBER = {4412543},
       DOI = {10.1017/9781108899727},
       URL = {https://doi.org/10.1017/9781108899727},
}

@book {MR1475463,
    AUTHOR = {Lothaire, M.},
     TITLE = {Combinatorics on words},
    SERIES = {Cambridge Mathematical Library},
      NOTE = {With a foreword by Roger Lyndon and a preface by Dominique
              Perrin,
              Corrected reprint of the 1983 original, with a new preface by
              Perrin},
 PUBLISHER = {Cambridge University Press, Cambridge},
      YEAR = {1997},
     PAGES = {xviii+238},
      ISBN = {0-521-59924-5},
   MRCLASS = {68R15 (03D03 03D40 05A99 20M05)},
  MRNUMBER = {1475463},
       DOI = {10.1017/CBO9780511566097},
       URL = {https://doi.org/10.1017/CBO9780511566097},
}

@book {MR576061,
    AUTHOR = {Rowen, Louis Halle},
     TITLE = {Polynomial identities in ring theory},
    SERIES = {Pure and Applied Mathematics},
    VOLUME = {84},
 PUBLISHER = {Academic Press, Inc. [Harcourt Brace Jovanovich, Publishers],
              New York-London},
      YEAR = {1980},
     PAGES = {xx+365},
      ISBN = {0-12-599850-3},
   MRCLASS = {16A38},
  MRNUMBER = {576061},
MRREVIEWER = {S.\ A.\ Amitsur},
}

@book {MR366968,
    AUTHOR = {Procesi, Claudio},
     TITLE = {Rings with polynomial identities},
    SERIES = {Pure and Applied Mathematics},
    VOLUME = {17},
 PUBLISHER = {Marcel Dekker, Inc., New York},
      YEAR = {1973},
     PAGES = {viii+190},
   MRCLASS = {16A38},
  MRNUMBER = {366968},
MRREVIEWER = {W.\ S.\ Martindale, III},
}

@book {MR1108620,
    AUTHOR = {Kemer, Aleksandr Robertovich},
     TITLE = {Ideals of identities of associative algebras},
    SERIES = {Translations of Mathematical Monographs},
    VOLUME = {87},
      NOTE = {Translated from the Russian by C. W. Kohls},
 PUBLISHER = {American Mathematical Society, Providence, RI},
      YEAR = {1991},
     PAGES = {vi+81},
      ISBN = {0-8218-4548-9},
   MRCLASS = {16R10},
  MRNUMBER = {1108620},
MRREVIEWER = {E.\ Formanek},
       DOI = {10.1090/mmono/087},
       URL = {https://doi.org/10.1090/mmono/087},
}

@article {MR302689,
    AUTHOR = {Formanek, Edward},
     TITLE = {Central polynomials for matrix rings},
   JOURNAL = {J. Algebra},
  FJOURNAL = {Journal of Algebra},
    VOLUME = {23},
      YEAR = {1972},
     PAGES = {129--132},
      ISSN = {0021-8693},
   MRCLASS = {16A42},
  MRNUMBER = {302689},
MRREVIEWER = {V.\ C.\ Cateforis},
       DOI = {10.1016/0021-8693(72)90050-6},
       URL = {https://doi.org/10.1016/0021-8693(72)90050-6},
}

@article{ART_2026__3_2_165_0,
     author = {Cerulli Irelli, Giovanni and De Loera Ch\'avez, Javier and Pascucci, Elena},
     title = {Quivers with {Polynomial} {Identities}},
     journal = {Annals of Representation Theory},
     pages = {165--177},
     year = {2026},
     publisher = {The Publishers of ART},
     volume = {3},
     number = {2},
     doi = {10.5802/art.38},
     language = {en},
     url = {https://art.centre-mersenne.org/articles/10.5802/art.38/}
}

@article {MR3556290,
    AUTHOR = {Sam, Steven V. and Snowden, Andrew},
     TITLE = {Gr\"obner methods for representations of combinatorial
              categories},
   JOURNAL = {J. Amer. Math. Soc.},
  FJOURNAL = {Journal of the American Mathematical Society},
    VOLUME = {30},
      YEAR = {2017},
    NUMBER = {1},
     PAGES = {159--203},
      ISSN = {0894-0347,1088-6834},
   MRCLASS = {18A25 (05E10 06A06 13C60 13P10 18C50 68Q70)},
  MRNUMBER = {3556290},
MRREVIEWER = {Geoffrey\ M. L. Powell},
       DOI = {10.1090/jams/859},
       URL = {https://doi.org/10.1090/jams/859},
}

@article {MR5062649,
    AUTHOR = {Berele, Allan and Cerulli Irelli, Giovanni and De Loera
              Ch\'avez, Javier and Pascucci, Elena},
     TITLE = {Polynomial identities for quivers via incidence algebras},
   JOURNAL = {Bull. Lond. Math. Soc.},
  FJOURNAL = {Bulletin of the London Mathematical Society},
    VOLUME = {58},
      YEAR = {2026},
    NUMBER = {5},
     PAGES = {Paper No. e70369, 4},
      ISSN = {0024-6093,1469-2120},
   MRCLASS = {05E10 (16G20 16R10)},
  MRNUMBER = {5062649},
       DOI = {10.1112/blms.70369},
       URL = {https://doi.org/10.1112/blms.70369},
}
\end{document}